\documentclass[a4paper,11pt]{article}

\usepackage{amsmath}
\usepackage{amssymb}
\usepackage{amsthm}
\usepackage{graphicx}
\usepackage{amscd}
\usepackage{epic, eepic}
\usepackage{url}
\usepackage{color}
\usepackage[utf8]{inputenc} 
\usepackage{comment}
\usepackage{enumerate}   
\usepackage{todonotes}
\usepackage{graphicx}
\usepackage{epstopdf}
\usepackage{enumitem}
\usepackage{tikz}
\usepackage{parskip}
\usepackage[top=25mm, bottom=25mm, left=28mm, right=28mm]{geometry}
\usepackage[bookmarks=false,hidelinks]{hyperref}
\usepackage{bbm}

\makeatletter
\renewenvironment{proof}[1][\proofname] {\par\pushQED{\qed}\normalfont\topsep6\p@\@plus6\p@\relax\trivlist\item[\hskip\labelsep\bfseries#1\@addpunct{.}]\ignorespaces}{\popQED\endtrivlist\@endpefalse}
\makeatother

\newtheorem{theorem}{\bf Theorem}[section]
\newtheorem{lemma}[theorem]{\bf Lemma}
\newtheorem{df}[theorem]{\bf Definition}
\newtheorem{corollary}[theorem]{\bf Corollary}

\newtheorem{problem}[theorem]{\bf Problem}

\theoremstyle{definition}

\def\eps{\varepsilon}

\newcommand{\trace}{\text{tr}}

\title{Short monochromatic odd cycles}

\author{Oliver Janzer\thanks{Institute of Mathematics, EPFL, Lausanne, Switzerland. Email: \textbf{oliver.janzer@epfl.ch}.}
	\and
Fredy Yip\thanks{Trinity College, University of Cambridge, United Kingdom. Email: \textbf{fy276@cam.ac.uk}.}}

\date{}

\begin{document}

\maketitle

\begin{abstract}
    It is easy to see that every $k$-edge-colouring of the complete graph on $2^k+1$ vertices contains a monochromatic odd cycle. In 1973, Erd\H os and Graham asked to estimate the smallest $L(k)$ such that every $k$-edge-colouring of $K_{2^k+1}$ contains a monochromatic odd cycle of length at most $L(k)$. Recently, Gir\~ao and Hunter obtained the first nontrivial upper bound by showing that $L(k)=O(\frac{2^k}{k^{1-o(1)}})$, which improves the trivial bound by a polynomial factor. We obtain an exponential improvement by proving that $L(k)=O(k^{3/2}2^{k/2})$. Our proof combines tools from algebraic combinatorics and approximation theory.
\end{abstract}

\section{Introduction}

Ramsey theory is the branch of combinatorics concerned with finding ordered substructures in large, potentially disordered objects. In graph Ramsey theory, we are usually interested in finding certain monochromatic subgraphs in edge-coloured graphs. This is an area that has seen some remarkable breakthroughs in the last few years; see \cite{conlon2021lower,campos2023exponential,mattheus2024asymptotics,balister2024upper}, for example.

In all of these recent breakthroughs, the number of colours is fixed as the other parameters grow. In general, our understanding of graph Ramsey problems in which the number of colours grows is rather poor. For example, the notorious Schur--Erd\H os problem asks whether the $k$-colour Ramsey number of the triangle (denoted as $R_k(K_3)$) grows exponentially in $k$, and this is wide open. More generally, despite much attention \cite{EG73,bondy1973ramsey,day2017multicolour,lin2019new}, we do not know the order of growth of the $k$-colour Ramsey number $R_k(C_{2\ell+1})$ of any fixed odd cycle $C_{2\ell+1}$.

The case of even cycles is easier, as it follows from the best upper bound on the Tur\'an number of $C_{2\ell}$ that $R_k(C_{2\ell})=O_{\ell}(k^{\frac{\ell}{\ell-1}})$. Concerning the case of a fixed number of colours and growing cycle length, Jenssen and Skokan \cite{jenssen2021exact} proved that for all fixed $k$ and sufficiently large $\ell$, we have $R_k(C_{2\ell+1})=2^k\ell+1$.

In this paper we study an old problem of Erd\H os and Graham \cite{EG73} that closely resembles the notorious multicolour Ramsey problem for odd cycles. It is easy to see that the complete graph on $2^k$ vertices can be $k$-edge-coloured so that each colour class is bipartite, but in any $k$-edge-colouring of the complete graph on $2^k+1$ vertices there exists a monochromatic odd cycle. Motivated by this, in 1975, Erd\H os and Graham \cite{EG73} asked to estimate the smallest $L(k)$ such that every $k$-edge-colouring of the complete graph on $2^k+1$ vertices has a monochromatic odd cycle of length at most $L(k)$.

\begin{problem}[Erd\H os and Graham \cite{EG73}]
    What is the smallest $L(k)$ such that every $k$-edge-colouring of the complete graph on $2^k+1$ vertices has a monochromatic odd cycle of length at most $L(k)$?
\end{problem}

This question also appears as Problem 75 in \cite{chung1997open}, as Problem 609 on  \url{https://www.erdosproblems.com/609} and on the webpage \url{https://mathweb.ucsd.edu/~erdosproblems/}.

Chung \cite{chung1997open} asked whether $L(k)$ is unbounded. This was answered affirmatively by Day and Johnson.

\begin{theorem}[Day and Johnson \cite{day2017multicolour}]
    We have $L(k)\geq 2^{\Omega{(\sqrt{\log k})}}$.
\end{theorem}

The first nontrivial upper bound was obtained recently by Gir\~ao and Hunter.

\begin{theorem}[Gir\~ao and Hunter \cite{girao2024monochromatic}]
    For every $\varepsilon>0$, there is $k_0>0$ such that we have $L(k)\leq \frac{2^k+1}{k^{1-\varepsilon}}$ for all $k\geq k_0$.
\end{theorem}

Note that this improved the trivial upper bound $2^k+1$ by a polynomial.
Our main result is an exponential improvement of the upper bound.

\begin{theorem} \label{thm:short mono odd cycle}
    In every $k$-edge-colouring of the complete graph on $2^k+1$ vertices, there exists a monochromatic odd cycle of length $O(k^{3/2}2^{k/2})$. That is, $L(k)=O(k^{3/2}2^{k/2})$.
\end{theorem}

In fact we will deduce Theorem \ref{thm:short mono odd cycle} from a more general result that applies to colourings of the complete graph on more vertices as well.

\begin{theorem} \label{thm:short mono odd cycle general}
    Let $0<\delta\leq 1$ such that $n=(1+\delta)2^k$ is an integer. Then in every $k$-edge-colouring of $K_n$ there is a monochromatic odd cycle of length at most $4k^{3/2}\delta^{-1/2}$.
\end{theorem}

Note that Theorem \ref{thm:short mono odd cycle general} implies Theorem \ref{thm:short mono odd cycle} by taking $\delta=2^{-k}$. It also improves a result of Gir\~ao and Hunter \cite{girao2024monochromatic}, who gave a bound $O(k^2\delta^{-1})$ on the length of the shortest monochromatic odd cycle in this setting.

\subsection{Proof overview}

Although our proof is short, we provide a brief outline of it. For simplicity, we will discuss the proof of the less general Theorem \ref{thm:short mono odd cycle} (the proof of Theorem \ref{thm:short mono odd cycle general} is almost identical). Our idea is to find a graph parameter $f(G)$ (taking nonnegative real values) that satisfies the following properties.
\begin{enumerate}
    \item $f(K_n)=n$, \label{item:complete}
    \item if $G_1$ and $G_2$ are graphs on the same vertex set and $G_1\cup G_2$ is the graph with edge set $E(G_1)\cup E(G_2)$, then $f(G_1\cup G_2)\leq f(G_1)f(G_2)$, and \label{item:union}
    \item if $G$ is an $n$-vertex graph that has no odd cycle of length at most $g$, then $f(G)\leq 2+\eps_{n,g}$, where $\eps_{n,g}$ is close to $0$ if $g$ is large. \label{item:odd girth}
\end{enumerate}

Assuming that such a parameter exists, let $n=2^k+1$ and consider a $k$-edge-colouring of $K_n$. Let $G_i$ be the graph formed by edges of the $i$th colour. Then, by properties \ref{item:union} and \ref{item:complete}, we have
$$\prod_{i=1}^k f(G_i)\geq f\left(\bigcup_{i=1}^k G_i\right)=f(K_n)=n=2^k+1.$$
On the other hand, if no $G_i$ contains an odd cycle of length at most $g$, then the left hand side is at most $(2+\eps_{n,g})^k$. By property \ref{item:odd girth}, if $g$ is large, then $\eps_{n,g}$ is close to $0$. But for very small values of $\eps_{n,g}$, we have $(2+\eps_{n,g})^k< 2^k+1$, which is a contradiction.

It remains to prove that a parameter with the desired properties indeed exists. We will show that we can take $f$ to be the Lov\'asz theta function of the complement of $G$.

\section{Proof}

As discussed in the proof outline, we will use the Lov\'asz theta function of a graph. Lov\'asz \cite{Lovasz} introduced this important graph parameter in 1979 in order to bound the Shannon capacity of a graph.


\begin{df}
    For a graph $G$ on vertex set $[n]$, an orthonormal representation $U$ is a collection of unit vectors $u_1, \dots, u_n$ in a Euclidean space $V$ such that if $i\neq j\in [n]$ and $ij$ is not an edge in $G$, then $u_i\cdot u_j = 0$. 
\end{df}

\begin{df}[Lov\'asz theta function]
    For a graph $G$ on vertex set $[n]$, its Lovász number $\vartheta(G)$ is defined to be
    \begin{equation*}
        \min_{c, U} \max_{i\in [n]} (c\cdot u_i)^{-2},
    \end{equation*}
    where the minimum is taken over all orthonormal representations $U$ of $G$ and all unit vectors $c\in V$ (where $V$ is the Euclidean space in which the vectors $u_i$ live). 
\end{df}

We will now verify that $f(G)=\vartheta(\overline{G})$ indeed satisfies the three key properties from the proof outline. Property \ref{item:complete} is well-known and easy to verify.

\begin{lemma} \label{complete graph}
    We have $\vartheta(\overline{K_n}) = n$. 
\end{lemma}

Note that Lemma \ref{complete graph} follows from the ``sandwich theorem'' (see, e.g., \cite{Knuth94}), stating that $\omega(G)\leq \vartheta(\overline{G})\leq \chi(G)$ for every graph $G$.


The next lemma verifies property \ref{item:union} in an equivalent form.

\begin{lemma} \label{lem:theta intersect}
    For graphs $G_1, G_2$ on vertex set $[n]$, let $G_1\cap G_2$ be the graph on vertex set $[n]$ with edge set $E(G_1)\cap E(G_2)$. The Lovász number is sub-multiplicative in the sense that
    $$\vartheta(G_1\cap G_2)\leq \vartheta(G_1)\vartheta(G_2).$$
\end{lemma}

\begin{proof}
    Let 
    \begin{align*}
        \vartheta(G_1) = \max_{i\in [n]} (c_1\cdot (U_1)_i)^{-2}, \\
        \vartheta(G_2) = \max_{i\in [n]} (c_2\cdot (U_2)_i)^{-2}, 
    \end{align*}
    for orthonormal representations $U_1, U_2$ of $G_1, G_2$ on Euclidean spaces $V_1, V_2$ and unit vectors $c_1\in V_1, c_2\in V_2$. Here $(U_1)_i, (U_2)_i$ denote the unit vector in $U_1, U_2$ corresponding to vertex $i$, respectively. Let $V = V_1\otimes V_2$ be the tensor product of $V_1$ and $V_2$. Let $U$ be a collection $u_1, \dots, u_n$ of unit vectors on $V$ given by $u_i = (U_1)_i\otimes (U_2)_i$. Notice that $U$ is an orthonormal representation of $G_1\cap G_2$. Indeed, if $i\neq j$ and $ij$ is not an edge in $G_1\cap G_2$, then we have either $ij\not \in G_1$ in which case $(U_1)_i\cdot (U_1)_j=0$, or $ij\not \in G_2$, in which case $(U_2)_i\cdot (U_2)_j=0$. In both cases we have $u_i\cdot u_j=((U_1)_i\cdot (U_1)_j)((U_2)_i\cdot (U_2)_j)=0$.
    
    Let $c = c_1\otimes c_2$. Then, for any $i\in [n]$, 
    \begin{equation*}
        c\cdot u_i = (c_1 \cdot (U_1)_i)(c_2 \cdot (U_2)_i). 
    \end{equation*}
    Hence, 
    \begin{align*}
        (c\cdot u_i)^{-2} &= (c_1 \cdot (U_1)_i)^{-2}(c_2 \cdot (U_2)_i)^{-2}\\
        &\leq \vartheta(G_1)\vartheta(G_2), 
    \end{align*}
    for any $i\in [n]$. Hence, 
    \begin{equation*}
        \vartheta(G_1\cap G_2)\leq \max_{i\in [n]} (c\cdot u_i)^{-2}\leq \vartheta(G_1)\vartheta(G_2),
    \end{equation*}
    as desired.
\end{proof}

Applying Lemma \ref{lem:theta intersect} with $\overline{G}_1$ and $\overline{G}_2$ in place of $G_1$ and $G_2$, we obtain the following corollary.

\begin{corollary} \label{edge-union}
    For graphs $G_1, G_2$ on vertex set $[n]$, let $G_1\cup G_2$ be the graph on vertex set $[n]$ with edge set $E(G_1)\cup E(G_2)$. Then
    $$\vartheta(\overline{G_1\cup G_2})\leq \vartheta(\overline{G_1})\vartheta(\overline{G_2}).$$
\end{corollary}

In \cite{Lovasz}, Lov\'asz provided several equivalent descriptions of the theta function, which is one of the reasons why this function is so useful. We will need the following characterization.

\begin{lemma}[Lovász \cite{Lovasz}] \label{alt def}
    Given an orthonormal representation $U$ of $G$, let $M(U)$ denote its Gram matrix given by $M(U)_{ij} = u_i\cdot u_j$. Let $\lambda_1(M(U))$ denote its largest eigenvalue. Then $\vartheta(\overline{G}) = \max_{U}\lambda_1(M(U))$, where the maximum is over all orthonormal representations of~$G$. 
\end{lemma}

We will also need the following well-known properties of the Chebyshev polynomial of the first kind.

\begin{lemma}
    If $g$ is an odd positive integer, then the degree $g$ Chebyshev polynomial of the first kind, denoted as $T_g$, satisfies the following properties.

    \begin{itemize}
        \item $T_g$ is an odd polynomial, i.e. it only contains monomials with odd exponents.
        \item $T_g(x)\geq -1$ for every $x\geq -1$.
        \item $T_g(x)=\frac{1}{2}\left(\left(x-\sqrt{x^2-1}\right)^g+\left(x+\sqrt{x^2-1}\right)^g\right)$ for every $x\geq 1$.
    \end{itemize}
\end{lemma}

We are now ready to prove that the complement of the theta function has property \ref{item:odd girth} from the proof outline. In other words, we prove that if $G$ is an $n$-vertex graph that does not contain an odd cycle of length at most $g$, then $\vartheta(\overline{G})$ is small. Under these assumptions, Alon and Kahale \cite{alon1998approximating} proved that $\vartheta(\overline{G})\leq 1+(n-1)^{1/g}$. While their result is tight up to a multiplicative constant depending on $g$ (and hence provides a good bound for \emph{constant} $g$), it would be too weak in our setting, where $g$ is large. Hence, we prove a different bound, using a novel argument.

\begin{lemma} \label{girth result}
    Let $g$ be an odd positive integer and let $G$ be a graph on vertex set $[n]$ which does not contain any odd cycle of length at most $g$. Then $\vartheta(\overline{G})\leq 2 + \frac{1}{2}\left((2n - 2)^{1/g} - 1\right)^2$. 
\end{lemma}

\begin{proof}
    We use the characterisation of $\vartheta(\overline{G})$ given by Lemma \ref{alt def}. For any orthonormal representation $U$ of $G$, let $B = M(U) - I$ be the difference between its Gram matrix and the identity matrix. Note that $B$ has vanishing diagonal, and for $i\neq j\in [n]$, $B_{ij} = 0$ whenever $ij$ is not an edge in $G$. Since $G$ does not contain any odd cycle of length at most $g$, it follows that for each odd $\ell\leq g$, $G$ does not contain a closed walk of length $\ell$, and therefore $\trace (B^{\ell}) = 0$. Let $\mu_n\leq \cdots\leq \mu_1$ be the eigenvalues of the real symmetric matrix $B$. The largest eigenvalue of $M(U)$ is $\mu_1+1$, so, by Lemma \ref{alt def}, it suffices to show that $\mu_1\leq 1 + \frac{1}{2}\left((2n - 2)^{1/g} - 1\right)^2$. We have
    \begin{equation*}
        \sum_{i = 1}^n \mu_i^{\ell} = \trace(B^{\ell})= 0, 
    \end{equation*}
    for any odd $\ell\leq g$. Therefore for any odd polynomial $p$ of degree at most $g$, 
    \begin{equation*}
        \sum_{i = 1}^n p(\mu_i) = 0. 
    \end{equation*}
    We take $p = T_g$ to be the $g$th Chebyshev polynomial of the first kind. As the Gram matrix $M(U) = B + I$ is positive semidefinite, we have $-1\leq \mu_n\leq \cdots\leq \mu_1$. Noting that $T_g(x)\geq -1$ for all $x\geq -1$, we have
    \begin{equation*}
        T_g(\mu_1) = -\sum_{i = 2}^n T_g(\mu_i)\leq n - 1. 
    \end{equation*}
    We have
    $$T_g(x)=\frac{1}{2}\left(\left(x-\sqrt{x^2-1}\right)^g+\left(x+\sqrt{x^2-1}\right)^g\right)\geq \frac{1}{2}\left(1+\sqrt{x^2-1}\right)^g$$
    for every $x\geq 1$. Therefore, assuming $\mu_1\geq 1$ (otherwise we are already done),
    \begin{equation*}
        \frac{1}{2}\left(1+\sqrt{\mu_1^2-1}\right)^g\leq T_g(\mu_1)\leq n - 1. 
    \end{equation*}
    Hence, 
    \begin{equation*}
        \mu_1\leq \sqrt{1 + \left((2n - 2)^{1/g} - 1\right)^2}\leq 1 + \frac{1}{2}\left((2n - 2)^{1/g} - 1\right)^2. 
    \end{equation*}
    As a result, $\lambda_1(M(U)) = \mu_1 + 1\leq 2 + \frac{1}{2}\left((2n - 2)^{1/g} - 1\right)^2$ for any orthonormal representation $U$ of $G$. Thus, $\vartheta(\overline{G})\leq 2 + \frac{1}{2}\left((2n - 2)^{1/g} - 1\right)^2$. 
\end{proof}

We are now ready to prove Theorem \ref{thm:short mono odd cycle general} in the following equivalent form.

\begin{theorem} 
    Let $g$ be an odd positive integer, let $0 \leq \delta <1$ such that $n = (1+\delta)2^k$ is an integer, and let $G_1, \dots, G_k$ be $n$-vertex graphs of odd girth greater than $g$ partitioning the edge set of the complete graph $K_n$. Then $g\leq 4k^{3/2}\delta^{-1/2}$. 
\end{theorem}

\begin{proof}
    Note that $K_n$ is the edge-union of $G_1, \dots, G_k$. Therefore recursively applying Corollary \ref{edge-union}, we have
    \begin{equation*}
        \vartheta(\overline{K_n})\leq \vartheta(\overline{G_1})\cdots\vartheta(\overline{G_k}). 
    \end{equation*}
    By Lemma \ref{girth result}, $\vartheta(\overline{G_i})\leq 2 + \eps_{n,g}$ for each $i = 1,\dots, k$, where $\eps_{n,g}= \frac{1}{2}\left((2n - 2)^{1/g} - 1\right)^2$. By Lemma \ref{complete graph}, $\vartheta(\overline{K_n}) = n = (1+\delta) 2^k$. Hence, 
    \begin{equation*}
        (1+\delta)2^k\leq (2+\eps_{n,g})^k. 
    \end{equation*}
    Therefore,
    \begin{equation*}
        1+\delta\leq (1+\eps_{n,g}/2)^k\leq \exp(k\eps_{n,g}/2), 
    \end{equation*}
    so, using the estimate $\exp(\delta/2)\leq 1+\delta$ which is valid since $0<\delta\leq 1$, we obtain
    $\delta\leq k\eps_{n,g}$. Plugging in the definition of $\eps_{n,g}$, we have
    $$\delta/k\leq \frac{1}{2}\left((2n - 2)^{1/g} - 1\right)^2,$$
    so
    $$\sqrt{2\delta/k}\leq (2n - 2)^{1/g} - 1.$$
    Noting that $2^{x/2} \leq 1 + x$ for $0\leq x\leq 2$, we have $2^{\sqrt{\delta/2k}}\leq 1+\sqrt{2\delta/k}$, and so
    $$2^{g\sqrt{\delta/2k}}\leq 2n-2\leq 2^{2k}.$$
    Hence,
    $$g\sqrt{\delta/2k}\leq 2k,$$
    which implies $g\leq 4k^{3/2}\delta^{-1/2}$, as desired.
\end{proof}

\section{Concluding remarks}

It is natural to wonder whether our results can be strengthened by obtaining a better bound than Lemma \ref{girth result} for the theta function of complements of graphs of high odd girth. It turns out that no significant improvement can be obtained this way. In fact, one cannot even improve Lemma \ref{girth result} significantly for the particular graph $C_g$, which has odd girth $g$ when $g$ is odd. Indeed, it is known \cite{Knuth94} that for every odd $g$, we have $\vartheta(\overline{C_g})=2+\frac{\pi^2}{2g^2}+O(g^{-4})$, and even if we could replace the upper bound in Lemma \ref{girth result} with $2+\frac{\pi^2}{2g^2}$, this would only improve our bound in Theorem~\ref{thm:short mono odd cycle} by a polynomial factor.

The reader might also wonder whether we can choose a graph parameter $f$ that is smaller than the complement of the theta function and has $f(C_g)-2\ll \vartheta(\overline{C_g})-2$ (so that the above issue does not arise), and yet it still satisfies the first two properties required in the proof overview. However, such a parameter does not exist. Indeed, for any such $f$ we would have
\begin{equation}
    g=f(K_g)\leq f(C_g)f(\overline{C_g})\leq \vartheta(\overline{C_g})\vartheta(C_g)=g, \label{eqn:other f}
\end{equation}
where the last equality follows from the fact (see \cite{Lovasz}) that $\vartheta(G)\vartheta(\overline{G})=n$ holds for every vertex-transitive $n$-vertex graph $G$. In particular, equality must hold everywhere in equation (\ref{eqn:other f}) and we have $f(C_g)=\vartheta(\overline{C_g})$.

On the other hand, bounds on the Shannon capacity $\Theta$ for all $n$-vertex graphs whose complements have large odd girth directly translate to bounds on our Ramsey problem. A similar connection (between Shannon capacities of complements of odd cycles and the Ramsey problem we study) was mentioned in a paper of Zhu \cite{zhu2025improved}, building on closely related earlier observations of Erd\H os, McEliece and Taylor \cite{erdHos1971ramsey} and Alon and Orlitsky \cite{alon1995repeated}. Indeed, consider a $k$-edge colouring of $K_n$ so that the colour classes define $n$-vertex graphs $G_1,\dots,G_k$, none of which contain an odd cycle of length at most $g$. Then
$\alpha(\overline{G_1}\boxtimes \overline{G_2}\boxtimes \dots \boxtimes \overline{G_k})\geq n$, where $\boxtimes$ denotes the strong product of graphs and $\alpha$ denotes the independence number. Indeed, when each $G_i$ is identified with the graph formed by the $i$th colour in our $k$-colouring of $K_n$, then the ``diagonal'' $\{(v,v,\dots,v): v\in V(K_n)\}$ is an independent set in $\overline{G_1}\boxtimes \overline{G_2}\boxtimes \dots \boxtimes \overline{G_k}$. Now let $G$ be the disjoint union of graphs $G_1, \dots, G_k$. Clearly $G$ is a graph on $kn$ vertices which does not contain an odd cycle of length at most $g$. But $\overline{G_1}\boxtimes \overline{G_2}\boxtimes \dots \boxtimes \overline{G_k}$ is an induced subgraph of $\overline{G}^k$ (the $k$-th strong power of $\overline{G}$), so we have
$$\Theta(\overline{G})^k\geq \alpha(\overline{G}^k)\geq \alpha(\overline{G_1}\boxtimes \overline{G_2}\boxtimes \dots \boxtimes \overline{G_k})\geq n.$$
Hence, a sufficiently good upper bound for the Shannon capacity of $kn$-vertex graphs whose complements contain no short odd cycles would mean that $g$ cannot be too large, which corresponds to an upper bound on the length of the shortest monochromatic odd cycle in any $k$-colouring of $K_n$. Since the Shannon capacity of every graph is upper bounded by the theta function of the same graph, this approach could lead to an improved bound in our main results. However, the best known upper bound for the Shannon capacity of the complement of an odd cycle comes from the theta function, so we encounter the same barrier as discussed in the first paragraph of this section.


\bibliographystyle{abbrv}
\bibliography{mybib}

\end{document}